\documentclass{article}

\usepackage{graphicx,xcolor,float}
\usepackage{hyperref,enumerate,enumitem}
\usepackage{amsthm,amsmath,amssymb}
\usepackage{tikz,comment}
\usepackage{caption,subcaption}
\usepackage[T1]{fontenc}

\newtheorem{theorem}{Theorem}[section]
\newtheorem{lemma}[theorem]{Lemma}
\newtheorem{definition}{Definition}[section]
\newtheorem{corollary}{Corollary}[theorem]
\newtheorem{claim}{Claim}[theorem]
\newtheorem{proposition}[theorem]{Proposition}

\newtheorem*{observation}{Observation}

\newtheorem*{remark}{Remark}
\newtheorem*{question}{Question}

\title{Dimension of Generic Reals}
\author{Yiping Miao\footnote{The author wants to thank Theodore Slaman, Sean Cody and many other people (especially various people the author talked to in Singapore Summer School 2025) for the helpful conversations and questions around the paper.}}
\date{}

\begin{document}

\maketitle

\begin{abstract}
This paper investigates the Hausdorff measure of certain sets of generics in computability theory. Let $\Gamma$ be the Turing ideal in which we take the dense open sets. The set of $\Gamma$-Cohen generics  has measure positive if and only if the gauge function is not dominated by every element in $\Gamma$, under some mild restrictions on the gauge function. The set of $\Gamma$-Mathias generics and the set of $\Gamma$-Sacks generics have measure positive if and only if the gauge function eventually dominates every element in $\Gamma$. This gives some comparison between the behavior of reals in the set and the measure of the set.
\end{abstract}

\tableofcontents

\section{Introduction}

There are two orthogonal notions of smallness: null sets and meager sets. A set is \textit{comeager} if it contains a countable intersection of dense open sets, a set is \textit{meager} if its complement is comeager. A typical element of a \textit{conull} set (complement of a null set) is a random real, while a typical element of a comeager set is a generic real.

There are comeager null sets. For example, a real $x$ is \textit{$\omega$-generic} if it is contained in every dense open set that is arithmetic definable. The set of $\omega$-generics is comeager and has Hausdorff dimension $0$. A natural question to ask is if we can characterize the `smallness' of this set. There are previous research on Liouville numbers characterizing the gauge profile of the set of Liouville numbers\cite{olsen2005exacti,olsen2006exact}. We follow the same line of research for $\omega$-generics.

We can consider a broader measure notion called gauge measures. A function $f:\mathbb{R}^+\rightarrow\mathbb{R}^+$ is a \textit{gauge function} if it is non-decreasing, right-continuous and $\lim_{x\rightarrow 0+}f(x)=0$. It takes in some $x$ and outputs the measure of a `$f$-dimension' basic open set with diameter $x$. For example, $f(x)=\pi x^2/4$ is a gauge function that takes in $x$ and outputs the measure of a $2$-dimensional disc with diameter $x$, which is the classical measure notion in a plane. Gauge functions give us finer control of measure with regard to diameters.

Since we work in Cantor space $2^\omega$ in this paper, we only need to consider the value of the gauge function on $\{2^{-k}:k\in\omega\}$, so we can code a gauge function using a real. When we talk about the complexity of a gauge function, we mean the complexity of the real coding this gauge function.

There are various ways to compare two gauge functions $g_0,g_1$. In this paper we compare the literal values. 

\begin{definition}
For gauge functions $g_0,g_1$,
\begin{enumerate}
    \item $g_0$ dominates $g_1$ if for every $x$, $g_0(x)\geq g_1(x)$.
    \item $g_0$ eventually dominates $g_1$ ($g_0\geq^* g_1$) if there exists $\delta>0$ such that for every $x\in (0,\delta)$, $g_0(x)\geq g_1(x)$.
\end{enumerate}
\end{definition}

For every perfect set, there is a gauge function under which it has positive measure. Since there is a perfect set of $\omega$-generics, there are gauge functions under which the set of $\omega$-generics has positive measure. We give a characterization of such gauge functions.
\\[2mm]
\noindent {\bf Theorem \ref{Cohen_dim}.} $H^f(\mathcal{C}_\Gamma)>0$ if and only if for every gauge functions $g\in\Gamma$, $\hat{f}\not\leq^*g$.
\\[2mm]
Here we state it in a more general form. $\mathcal{C}_\Gamma$ is the set of $\omega$-generics when $\Gamma$ is the set of arithmetic reals. $\hat{f}$ is a slight modification of $f$ that originally appears in the 2006 paper by Olsen and Renfro\cite{olsen2006exact}.

The above theorem shows that the level of generocity determines the dimension.

An $\omega$-generic (viewed as a function $\omega\rightarrow\omega$) is not dominated by every arithmetic function. We can see that the characterization closely resembles the behavior of $\omega$-generics. The second half of the paper looks into Mathias forcing and Sacks forcing. These two forcing generics have opposite behaviors in the sense that Mathias forcing adds fast-growing functions and Sacks forcing adds slow-growing functions. We obtain the following results.
\\[2mm]
\noindent {\bf Theorem \ref{Mathias}.} $H^f(\mathcal{M}_\Gamma)>0$ if and only if for every gauge function $g\in\Gamma$, $f\geq^* g$.
\\[2mm]
\noindent {\bf Theorem \ref{Sacks}.} $H^f(S_\Gamma)>0$ if and only if for every gauge function $g\in\Gamma$, $f\geq^* g$.
\\[2mm]
This tells us despite $\Gamma$-Mathias generics and $\Gamma$-Sacks generics have radically different behaviors, under the gauge function perspective, these two sets of generics are indistinguishable. We pose a natural (and vague) question here.

\begin{question}
Is there a pattern between the behavior of reals and the behavior of gauge functions that makes them measure positive?
\end{question}

We put a few easy facts about gauge functions and perfect trees in the next section that we will reference a few times in our proofs. 

\subsubsection*{Forcing}

A \textit{forcing notion} is a poset $\mathbb{P}=(P,\leq_P)$, where $\leq_P$ is a partial order on $P$. We call every $t\in P$ conditions, and say $t$ \textit{extends} $s$ if $t\leq_P s$. We define the following notion which is a way to produce a dense set.

\begin{definition}
A function $\Theta:P\rightarrow P$ is a density operator if for all $t\in P$, $\Theta(t)\leq t$.
\end{definition}

Clearly $ran(\Theta)$ is a dense set. $\Theta$ serves as the Skolem function for a dense set so that we have a way(of certain complexity) to get from a condition to its extension.

We also need the following definition since our proof works in every complexity that has enough closure properties.

\begin{definition}
    A set $\Gamma$ of reals is a Turing ideal if
    \begin{enumerate}
        \item (downward closed) $\forall x,y\ [(x\in\Gamma\ \wedge\ y\leq_T x)\rightarrow y\in\Gamma],$
        \item (closed under joins) $\forall x,y\in\Gamma (x\oplus y\in\Gamma)$.
    \end{enumerate}
    If $\Gamma$ is a nonempty countable Turing ideal, we call it countable ideal for short.
\end{definition}

The set of recursive reals, the set of arithmetic reals, the set of hyperarithmetic reals, etc, are countable ideals. Every countable ideal contains the set of recursive reals. We fix this symbol $\Gamma$ for a countable ideal throughout the paper.

\subsubsection*{Geometric Measure Theory in Cantor Space}

For every gauge function $f:\mathbb{R}^+\rightarrow\mathbb{R}^+$ and every $A\subseteq 2^\omega$, define the $f$-(outer) measure as
$$H^f(A)=\lim_{n\rightarrow \infty}\inf\left\{\sum_i f(2^{-|s_i|}):\bigcup_i [s_i]\supseteq A\text{ and }\forall i(|s_i|\geq n)\right\}.$$

This is the gauge measure with respect to $f$. For standard properties of $H^f$, see \cite{rogers1998hausdorff}.

For every $x\in\omega^\omega$ that is non-decreasing and $\lim_{n\rightarrow\infty}x(n)=\infty$, we define a gauge function $f_x:\mathbb{R}^+\rightarrow\mathbb{R}^+$ by $x$ as
$$f_x(r)=2^{-x(n)},n \text{ is the least such that }2^{-n}\leq r.$$

It is easy to see that for every gauge function $f$, there is some $x\in \omega^\omega$ such that $f_x\leq f\leq 2f_x$ on a neighborhood of $0$, and we only care if the measure is zero, finite positive or infinite. From now on whenever we refer to a gauge function, we mean a gauge function defined by some $x\in\omega^\omega$, and whenever we talk about complexity of a gauge function, we mean the complexity of $x\in\omega^\omega$ defining it.

\subsubsection*{Notations and Conventions}

We naturally identify $2^\omega$ with $\mathcal{P}(\omega)$. $|A|$ is the cardinality of $A$.

There are three uses of $[*]$ here: For an element $s\in A^{<\omega}$, $[s]=\{x\in A^\omega:x\supset s\}$; for a tree $T\subseteq A^{<\omega}$, $[T]$ is the set of paths through $T$; for every Mathias condition $(s,A)$, $[s,A]$ is the basic open set it generates in Ellentuck topology.

\section{Perfect Trees and Their Measures}

We put some facts about gauge functions and perfect trees here.

\subsection{Perfect Trees}

For a tree $T\subseteq 2^{<\omega}$, we define $height(T)=\sup \{|t|+1:t\in T\}$.
\\[2mm]
For a perfect $T$ there is a unique index function $\sigma_T:2^{<\omega}\rightarrow T$ satisfying
\begin{enumerate}
    \item for all $t\in 2^{<\omega}$, $\sigma_T(t)\ \widehat{}\ 0\subseteq\sigma_T(t\ \widehat{}\ 0), \sigma_T(t)\ \widehat{}\ 1\subseteq\sigma_T(t\ \widehat{}\ 1)$,
    \item for all $t\in T$, there is some $s\in 2^{<\omega}$ such that $\sigma_T(s)\supseteq t$.
\end{enumerate}
And it induces a homeomorphism between $2^\omega$ and $[T]$.

We can consider a specific type of perfect tree where the tree always splits at the same level.

\begin{definition}
Let $T\subseteq 2^{<\omega}$ be a perfect tree and $\sigma_T: 2^{<\omega}\rightarrow T$ be the index function. $T$ is a uniformly perfect tree (a `uniform tree' for short) if
\begin{center}
for every $t_0,t_1\in 2^{<\omega}$ with $|t_0|=|t_1|$, $|\sigma_T(t_0)|=|\sigma_T(t_1)|$.
\end{center}
If $T$ is uniform, we let $\sigma_T(n)=\max \{|\sigma_T(s)|:|s|=n\}$.
\end{definition}

And we need the finite segment of it as well. The following definition can also be made in first-order similar to how we define uniform trees.

\begin{definition}
A tree $T\subseteq 2^{<\omega}$ is partial uniform if there is a uniform tree $S$ and some $n\in\omega$ such that
$$T=S\upharpoonright n.$$
\end{definition}

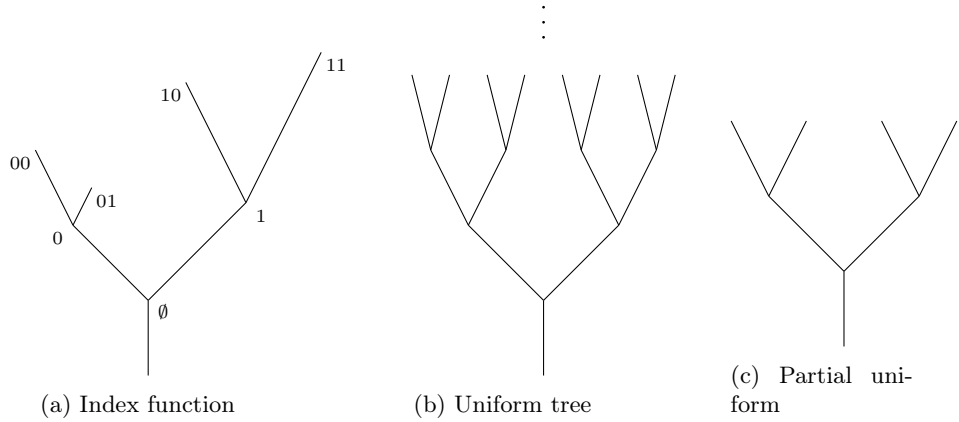
\begin{figure}[ht]
     \centering
     \begin{subfigure}[b]{0.3\textwidth}
         \centering
\begin{tikzpicture}
\draw (1.5,0) -- (1.5,1);
\draw (1.5,1) -- (0.5,2);
\draw (1.5,1) -- (2.8,2.3);
\draw (0.5,2) -- (0,3);
\draw (0.5,2) -- (0.75,2.5);
\draw (2.8,2.3) -- (2,3.9);
\draw (2.8,2.3) -- (3.8,4.3);
\node[label={[xshift = 0.2cm, yshift= -0.5cm]\scriptsize{$\emptyset$}}] (a) at (1.5,1) {};
\node[label={[xshift = -0.2cm, yshift= -0.5cm]\scriptsize{$0$}}] (b) at (0.5,2) {};
\node[label={[xshift = 0.2cm, yshift= -0.5cm]\scriptsize{$1$}}] (c) at (2.8,2.3) {};
\node[label={[xshift = -0.2cm, yshift= -0.5cm]\scriptsize{$00$}}] (d) at (0,3) {};
\node[label={[xshift = 0.2cm, yshift= -0.5cm]\scriptsize{$01$}}] (e) at (0.75,2.5) {};
\node[label={[xshift = -0.2cm, yshift= -0.5cm]\scriptsize{$10$}}] (f) at (2,3.9) {};
\node[label={[xshift = 0.2cm, yshift= -0.5cm]\scriptsize{$11$}}] (g) at (3.8,4.3) {};
\end{tikzpicture}         
         \caption{Index function}
         \label{indexfcn}
     \end{subfigure}
     \hfill
     \begin{subfigure}[b]{0.2\textwidth}
         \centering
\begin{tikzpicture}
\draw (0,0) -- (0,1);
\foreach \x in {-1,1}
\draw (0,1) -- (\x,2);
\foreach \x in {-1,1}
\foreach \y in {-0.5,0.5}
\draw (\x,2) -- (\x+\y,3);
\foreach \x in {-1,1}
\foreach \y in {-0.5,0.5}
\foreach \z in {-0.25,0.25}
\draw (\x+\y,3) -- (\x+\y+\z, 4);
\filldraw (0,4.5) circle (0.3pt);
\filldraw (0,4.7) circle (0.3pt);
\filldraw (0,4.9) circle (0.3pt);
\end{tikzpicture}
         \caption{Uniform tree}
         \label{Uniformtree}
     \end{subfigure}
\hfill
     \begin{subfigure}[b]{0.2\textwidth}
         \centering
\begin{tikzpicture}
\draw (0,0) -- (0,1);
\foreach \x in {-1,1}
\draw (0,1) -- (\x,2);
\foreach \x in {-1,1}
\foreach \y in {-0.5,0.5}
\draw (\x,2) -- (\x+\y,3);
\end{tikzpicture}
         \caption{Partial uniform}
         \label{paruni}
     \end{subfigure}
\end{figure}

\subsection{The Measure of a Uniform Tree}

We provide an easy way to calculate the measure of uniform trees. The proofs are routine so we omit them.

\begin{proposition}
For every gauge function $f$,
$$H^f(2^\omega)=\liminf_{n\rightarrow\infty}f(2^{-n})\cdot 2^n.$$
\end{proposition}

\begin{proposition}
For every uniform tree $T\subseteq 2^{<\omega}$ and gauge function $f$, define gauge function $g$ as follows
$$g(2^{-n})=f(2^{-\sigma_T(n)}).$$
Then $H^f([T])=H^g(2^\omega)$.
\end{proposition}

\begin{corollary}\label{uniform_measure}
For every uniform tree $T\subseteq 2^{<\omega}$ and gauge function $f$,
$$H^f([T])=\liminf_{n\rightarrow\infty}f(2^{-n})\cdot |\{s\in T:|s|=n\}|.$$
\end{corollary}

\subsection{How to Build Gauge Functions of Small Dimensions}

We want to define a gauge function $f_T$ so that the $f_T$ measure of $[T]$ would be large in a uniform way. We define $f_T$ as 
$$f_T(2^{-n})=2^{-m},\text{ where }m\in\omega\text{ is the least such that }\max\{|\sigma_T(t)|:|t|=m\}\geq n.$$

\begin{figure}[H]
    \centering
\begin{tikzpicture}
\draw (1.5,0) -- (1.5,1);
\draw (1.5,1) -- (0.5,2);
\draw (1.5,1) -- (2.8,2.3);
\draw (0.5,2) -- (0,3);
\draw (0.5,2) -- (0.75,2.5);
\draw (2.8,2.3) -- (2,3.9);
\draw (2.8,2.3) -- (3.8,4.3);
\node[label={[xshift = 0.2cm, yshift= -0.5cm]\scriptsize{$\emptyset$}}] (a) at (1.5,1) {};
\node[label={[xshift = -0.2cm, yshift= -0.5cm]\scriptsize{$0$}}] (b) at (0.5,2) {};
\node[label={[xshift = 0.2cm, yshift= -0.5cm]\scriptsize{$1$}}] (c) at (2.8,2.3) {};
\node[label={[xshift = -0.2cm, yshift= -0.5cm]\scriptsize{$00$}}] (d) at (0,3) {};
\node[label={[xshift = 0.2cm, yshift= -0.5cm]\scriptsize{$01$}}] (e) at (0.75,2.5) {};
\node[label={[xshift = -0.2cm, yshift= -0.5cm]\scriptsize{$10$}}] (f) at (2,3.9) {};
\node[label={[xshift = 0.2cm, yshift= -0.5cm]\scriptsize{$11$}}] (g) at (3.8,4.3) {};
\draw [dashed] (-0.9,1) -- (3.9,1);
\draw [dashed] (-0.9,2.3) -- (3.9,2.3);
\draw [dashed] (-0.9,4.3) -- (3.9,4.3);
\draw[<->] (-0.8,0.1) -- (-0.8,0.9);
\draw[<->] (-0.8,1.1) -- (-0.8,2.2);
\draw[<->] (-0.8,2.4) -- (-0.8,4.2);
\node[label={[xshift = -0.3cm, yshift= -0.3cm]\scriptsize{$1$}}] (a) at (-0.8,0.5) {};
\node[label={[xshift = -0.3cm, yshift= -0.3cm]\scriptsize{$1/2$}}] (a) at (-0.8,1.65) {};
\node[label={[xshift = -0.3cm, yshift= -0.3cm]\scriptsize{$1/4$}}] (a) at (-0.8,3.3) {};
\end{tikzpicture}
\end{figure}

\begin{proposition}\label{smalldim_lemma}
\begin{enumerate}
    \item $f_T$ is a gauge function.
    \item If $T,S\subseteq 2^{<\omega}$ are perfect trees and $n\in\omega$ such that
    $$\forall t\in 2^{<\omega}(|t|\leq n\rightarrow (t\in T\leftrightarrow t\in S)),$$
    then for all $k\leq n$, $f_T(2^{-k})=f_S(2^{-k})$.
    \item For every open cover $\bigcup_i [t_i]\supseteq [T]$,
    $$\sum_{i\in\omega} f_T(2^{-|t_i|})\geq 1.$$
\end{enumerate}
\end{proposition}

\section{Cohen Forcing}

\subsubsection*{The Dimension of the Set of \texorpdfstring{$\Gamma$}{Gamma}-Cohen Generics}

Let $\mathbb{P}_\mathcal{C}=(2^{<\omega},\supseteq)$ be the poset for Cohen forcing. We can view subsets of $2^{<\omega}$ as subsets of $\omega$.
\begin{definition}
Given a countable ideal $\Gamma$, $x\in 2^\omega$ is $\Gamma$-Cohen generic if it meets every dense set $D\in \Gamma$.
\end{definition}

Take $\Gamma$ be the set of arithmetic reals, $\Gamma$-Cohen generics are exactly $\omega$-generics in recursion theory.

\begin{lemma}\label{arith_Cohen}
Consider the set $\mathcal{C}_\Gamma$ of $\Gamma$-Cohen generics, for every gauge function $g$ such that $g(x)/x$ is monotonically decreasing, the following are equivalent:
\begin{enumerate}
    \item $H^g(\mathcal{C}_\Gamma)>0$,
    \item for every gauge functions $f\in \Gamma$, $f\not\geq^* g$.
\end{enumerate}
We do not need condition `$g(x)/x$ monotonically decreasing' to prove the implication 1 $\rightarrow$ 2.
\end{lemma}

\begin{proof}
If there is some $f\in\Gamma$ such that $g\leq^* f$, to show $H^g(\mathcal{C}_\Gamma)=0$, we only need to show $H^f(\mathcal{C}_\Gamma)=0$. Enumerate $2^{<\omega}$ computably as $\{\sigma_n:n\in\omega\}$. For every $l\in\omega$, we can define $O_l=\{\sigma_{k_n}:n\in\omega\}$ as

$k_n=$ the smallest $k$ satisfying $|\sigma_{k}|\geq l$, $\sigma_{k}\supseteq \sigma_n$ and $f(2^{-|\sigma_{k}|})\leq 2^{-(l+n+1)}$.\\
Then
\begin{enumerate}
    \item $O_l$ is computable and dense,
    \item $\sum_n f(2^{-|\sigma_{k_n}|})\leq 2^{-l}$.
\end{enumerate}
Since $O_l$ is computable, $O_l\in\Gamma$. This witnesses $H^f(C_\Gamma)=0$.\\[4mm]
For the other direction, without loss of generality let $g(1)=1$. 
\begin{observation}
If $H^g(\mathcal{C}_\Gamma)>0$, since $\mathcal{C}_\Gamma$ is $G_\delta$, there is a closed set $C\subseteq \mathcal{C}_\Gamma$ such that $H^g(C)>0$(\cite{BesicovitchA.S.1952Oeos,DaviesR.O.1952Sofm}).
By perfect set property, there is a perfect tree $T$ such that $[T]\subseteq C$ and $C\setminus [T]$ is countable, so $H^g([T])>0$. We will build such a perfect tree.
\end{observation}

We identify a subset of $2^{<\omega}$ with the open set it generates. Fix an enumeration $\{O_n\}_{n>0}$ of all elements of $\Gamma$ that generate dense open sets. If for every gauge function in $f\in\Gamma$, $g\not\leq^* f$, we will define a uniform tree $T=\bigcup T_n$ with the following properties:
\begin{enumerate}
    \item There exists an increasing sequence $\{l_n\}\subseteq\omega$ such that $T_n=T\upharpoonright l_n$.
    \item $g(2^{-n})\cdot |\{s\in T:|s|=n\}|\geq 1$ for every $n\in\omega$.
    \item For every end-node $s\in T_n$, $s\in O_n$.
\end{enumerate}
Then $[T]\subseteq C_\Gamma$ and by Corollary \ref{uniform_measure} we have $H^f([T])\geq 1$. We will define $T_n$ by recursion and the following claim establishes one step in the process.

\begin{claim}\label{one-step}
For every partial uniform tree $S\subseteq 2^{<\omega}$ satisfying for every $n\leq height(S)$,
$$g(2^{-n})\cdot |\{s\in S:|s|=n\}|\geq 1,$$
and every dense open set $O\in\Gamma$, there is a partial uniform tree $S^*\supset S$ such that
\begin{enumerate}
    \item Let $l=height(S)$, $S^*\upharpoonright l=S$.
    \item $g(2^{-n})\cdot |\{s\in S^*:|s|=n\}|\geq 1$, for all $n\leq height(S^*)$,
    \item for every end-node $s\in S^*$, $s\in O$.
\end{enumerate}
\end{claim}

\begin{proof}[Proof of the Claim]
Let $K=|\{s\in S: |s|=l\}|$. Enumerate $2^{<\omega}$ computably as $\{\sigma_n:n\in\omega\}$, define $\Theta(t)=\sigma_n$ where $n$ is the least such that
\begin{enumerate}
    \item $\sigma_n\supseteq t$,
    \item $\sigma_n\in O$.
\end{enumerate}
Let $k_{-1}=l$, for every $n\in\omega$, define
$$P_n=\{t:t\supseteq s\text{ where }s\text{ is an end-node of }S\text{ and }|t|=|s|+n\},$$
$$k_n=\max\{k_{n-1}+1,|\Theta(t)|:t\in P_n\}.$$
Define a gauge function $f$ as
$$f(2^{-c})=K^{-1}\cdot 2^{-n},\text{ where }n\text{ is the largest such that }k_n\leq c.$$
Since $f\in\Gamma$, there is some $c>l$ such that $g(2^{-c})\geq f(2^{-c})$. Let $n$ be the largest such that $k_n\leq c$, then $g(2^{-k_n})\geq K^{-1}\cdot 2^{-n}$. Let
$$S^*=\{\sigma\in 2^{<\omega}:\exists t\in P_n(\sigma,\Theta(t)\text{ comparable and }|\sigma|\leq k_n)\}.$$
We know
$$g(2^{-k_n})\cdot |\{s\in S^*:|s|=k_n\}|\geq 1.$$
For the assumption $g(x)/x$ decreasing and $g(2^{-l})\cdot |\{s\in S^*:|s|=l\}|\geq 1$ we know
$$g(2^{-{l+m}})\cdot |\{s\in S^*:|s|=l+m\}|\geq 1,\ \forall m\leq n.$$
There is no splitting on $S^*$ from level $l+n$ to level $k_n$, so
$$g(2^{-n})\cdot |\{s\in S^*:|s|=n\}|\geq 1,\ \forall n\leq height(S^*).$$
\end{proof}

Let $T_0=\emptyset$, given $T_n$, let $S=T_n$, $O=O_n$ and define $T_{n+1}=S^*$ from the claim. Note that $T=\bigcup T_n$ is a perfect tree since $\lim_{x\rightarrow 0+}g(x)=0$. This completes the proof.
\end{proof}

\subsubsection*{Lifting the Restriction}

The following proposition is Lemma 2.2 in \cite{olsen2006exact}. We sketch part of the proof since we are working in $2^\omega$ now.

\begin{proposition}\label{2006}(\cite{olsen2006exact})
If $g:\{2^{-n}:n\in\omega\}\rightarrow \{2^{-n}:n\in\omega\}$ is a gauge function, define function $\hat{g}:\{2^{-n}:n\in\omega\}\rightarrow \{2^{-n}:n\in\omega\}$ as
$$\hat{g}(x)=\inf_{0<s\leq x}x\frac{g(s)}{s}$$
\begin{enumerate}
    \item $\hat{g}$ is a gauge function.
    \item The function 
    $$x\mapsto \frac{\hat{g}(x)}{x}$$
    is decreasing in a neighborhood of $0$.
    \item \label{L3} We have
    $$H^{\hat{g}}(E) = H^g(E)$$
    for every $E\subseteq 2^\omega$.
\end{enumerate}
\end{proposition}

\begin{proof}
We only show the $H^g(E)\leq H^{\hat{g}}(E)$ direction in \ref{L3}. Take every $\delta>0$, we will show $H^{\hat{g}}_\delta(E)+\delta\geq H^g_\delta(E)$. Take every $\delta$ cover of $E$ as $([\sigma_i])_i$, let $r_i=2^{-|\sigma_i|}$. Since $\hat{g}(r_i)+\frac{\delta}{2^{i+1}}>\hat{g}(r_i)$, there is some $0<s_i\leq r_i$ such that $r_i\frac{g(s_i)}{s_i}\leq \hat{g}(r_i)+\frac{\delta}{2^{i+1}}$. Note that both $s_i,r_i\in \mathbb{D}$, and $[\sigma_i]$ can be divided into $\{[\sigma_{ij}]:j\leq \frac{r_i}{s_i}\}$. Then we have
$$\sum_i \hat{g}(r_i)+\delta\geq \sum_i r_i\frac{g(s_i)}{s_i}=\sum_i\sum_j g(2^{-[\sigma_{ij}]})\geq H^g_\delta(E).$$
So $H^{\hat{g}}_\delta(E)+\delta\geq H^g_\delta(E)$. Let $\delta\rightarrow 0+$, we get the desired result.
\end{proof}

\begin{remark}
$g(x)/x$ is decreasing if and only if $g=\hat{g}$.
\end{remark}

\begin{theorem}
\label{Cohen_dim}
$H^f(\mathcal{C}_\Gamma)>0$ if and only if for every gauge functions $g\in\Gamma$, $g\not\geq^*\hat{f}$.
\end{theorem}

\begin{proof}
By Lemma \ref{arith_Cohen} and Proposition \ref{2006}.
\end{proof}

\begin{remark}
It is not true if we replace $\hat{g}$ by $g$ in the above corollary. We can build some gauge function $g$ that for every $f\in\Gamma$, $g\not\leq^* f$ but $\hat{g}(x)\leq x$.
\end{remark}

\subsubsection*{Some Observations}

One natural question to ask is the boundary behavior: what value can $H^f(\mathcal{C}_\Gamma)$ be if it is positive. It turns out it can not be $\sigma$-finite. This is basically the same proof idea from \cite{olsen2006exact}.

\begin{proposition}\label{no_strong_dim}
Let $\Upsilon$ be a set of gauge functions such that
\begin{enumerate}
    \item for every $h\in\Upsilon$ and every $n\in\omega$, there exists $k\in\Upsilon$ such that $n\cdot h\leq^* k$.
    \item $\text{id}\in \Upsilon$.
\end{enumerate}
If $g$ is a gauge function such that
\begin{enumerate}
    \item for every $h\in\Upsilon$, $g\not\leq^* h$,
    \item $g(x)/x$ is monotonically decreasing,
\end{enumerate}
then there is a gauge function $f$ such that
\begin{enumerate}
    \item $f\succ g$,
    \item for every $h\in\Upsilon$, $f\not\leq^* h$,
    \item $f(x)/x$ is monotonically decreasing.
\end{enumerate}
\end{proposition}

\begin{corollary}
    If $H^g(\mathcal{C}_\Gamma)>0$, then $H^g(\mathcal{C}_\Gamma)$ has to be non $\sigma$-finite.
\end{corollary}

\begin{proof}
For every gauge function $g$ with $H^g(\mathcal{C}_\Gamma)>0$, if $H^g(\mathcal{C}_\Gamma)$ is $\sigma$-finite, $\mathcal{C}_\Gamma=\bigcup_n A_n$ where $H^g(A_n)<\infty$ for every $n$. So 
$$H^{\hat{g}}(A_n)<\infty.$$
From Proposition \ref{no_strong_dim} we have some gauge function $f\succ \hat{g}$ such that for every gauge functions $h\in\Gamma$, $f\not\leq^* h$. Then $$H^f(\mathcal{C}_\Gamma)\leq \sum_n H^f(A_n)=0,$$
which contradicts Theorem \ref{Cohen_dim}.
\end{proof}

From a theorem by Besicovitch\cite{besicovitch1956definition}, if $A$ is analytic and has the above property, then $A$ has no strong dimension. See also recent related work on strong dimensions\cite{slaman2024extending}.

\begin{corollary}
    There is no strong dimension for $\mathcal{C}_\Gamma$.
\end{corollary}

\subsubsection*{Random Real Forcing}

The other natural forcing to consider is random real forcing, as this is the dual notion of Cohen forcing.

The following proposition makes random reals less interesting for characterization since the set of random reals has Lebesgue measure $1$.

\begin{proposition}
For every $A\subseteq 2^\omega$ such that $\mu(A)=1$, and every gauge function $f$,
\begin{center}
$H^f(A)=0$ if and only if $\liminf_{x\rightarrow 0+} f(x)/x=0$.
\end{center}
\end{proposition}

\begin{proof}
$\Leftarrow$: If $\liminf_{x\rightarrow 0+} f(x)/x=0$, then there is a strictly increasing sequence $\{n_k\}_{k\in\omega}\subseteq\omega$ such that 
$$\lim_{k\rightarrow\infty}f(2^{-n_k})/(2^{-n_k})=0.$$
Take $\{\sigma_l\}_{l\in\omega}\subseteq 2^{<\omega}$ such that $A\subseteq \bigcup_l[\sigma_l]$, $\sum_l 2^{-|\sigma_l|}\leq 2$ and each $2^{-|\sigma_l|}=2^{-n_k}\leq \varepsilon$ for some $k$. Then 
$$\sum_l f(2^{-|\sigma_l|})\leq 2\varepsilon.$$
$\Rightarrow$: If $\liminf_{x\rightarrow 0+} f(x)/x\geq\varepsilon>0$, then there is some $y\in (0,1)$ such that for every $x\in (0,y)$, $f(x)/x\geq\varepsilon$. So $H^f(A)\geq \varepsilon\cdot\mu(A)>0$.
\end{proof}

\section{Mathias Forcing and Sacks Forcing}

\subsection{Mathias Forcing}

Mathias forcing adds generics as fast-growing functions (in $\omega^\omega$, if in $2^\omega$ this means the real has very sparse $1$'s). P. Cholak, D. Dzhafarov, J. Hirst and T. Slaman studied Mathias forcing in the context of computability theory\cite{cholak2014generics}. We will also define Mathias forcing with reduced complexity, but in a specific way so we have a correspondence in terms of geometric measure theory.\\[4mm]
Let $\mathbb{P}_\mathcal{M}=(P_\mathcal{M},\leq_\mathcal{M})$ be the poset for Mathias forcing, i.e.
    $$P_\mathcal{M}=\{(s,A)\in 2^{<\omega}\times \mathcal{P}(\omega):|A|=\aleph_0 \wedge |s|\leq \min A\}$$
    $$(s,A)\leq_\mathcal{M} (t,B)\Leftrightarrow s\supseteq t\ \wedge\ A\subseteq B\ \wedge\ s-t\subseteq B.$$
For every $(s,A)\in\mathbb{P}_M$ let 
    $$[s,A]=\{x\in 2^\omega:s=x\upharpoonright |s|\wedge x\subseteq s\cup A\}.$$
We say $x\in 2^\omega$ meets some $D\subseteq P_\mathcal{M}$ if there is some $(s,A)\in D$ such that $x\in [s,A]$.\\[2mm]
We will restrict ourselves to Mathias generics for a certain complexity (complexity $\Gamma$).

We identify each Mathias condition $(s,A)$ with a real in $2^\omega$ coding it, and say a density function $\Theta:P_\mathcal{M}\rightarrow P_\mathcal{M}$ is coded by a real $y$ if $\Theta$ can be extended to a continuous function from $2^\omega$ to $2^\omega$ and real $y$ codes this continuous function in some standard way.

\begin{definition}
Given a countable ideal $\Gamma$,
\begin{enumerate}
    \item $D\in\mathcal{D}_\Gamma$ if
    $$\exists x\in \Gamma (x\text{ codes }\Theta\ \wedge\ D=ran(\Theta\upharpoonright \Gamma)).$$
    \item $x\in 2^\omega$ is $\Gamma$-Mathias generic if $x$ meets every $D\in \mathcal{D}_\Gamma$.
\end{enumerate}
\end{definition}

\begin{theorem}
\label{Mathias}
$H^f(\mathcal{M}_\Gamma)>0$ if and only if for every gauge function $g\in\Gamma$, $f\geq^* g$.
\end{theorem}

\begin{proof}
$\Rightarrow$: If there is some $g\in \Gamma$ such that $f\not\geq^* g$, we want to find some $x$ coding $\Theta$ such that
$$\Theta:(s,A)\mapsto (s,B),$$
where $H^f([s,B])=0$ and $x\leq_T g$.

Since $[s,B]$ is a uniform tree, from Corollary \ref{uniform_measure} we know on level $n$ of the tree we only need to monitor the number of branches on the level and the value $g(2^{-n})$. The strategy is: do not split until $g$ drops under a certain value, then split.

We describe such an algorithm:
Suppose $B\upharpoonright n$ is defined and $l:=|\{k<n:B(k)=1\}|$.
$$B(n)=1\text{ if }A(n)=1\text{ and }g(2^{-(n+1)})\cdot 2^l< 2^{-l},$$
otherwise $B(n)=0$.

Since $f\not\geq^* g$, there are infinitely many $n\in\omega$ such that $f(2^{-n})<g(2^{-n})$.If $\lim_{n\rightarrow\infty}g(2^{-n})\cdot 2^{|\{k<n:B(k)=1\}|}=0$, then
$$\liminf_{n\rightarrow\infty}f(2^{-n})\cdot 2^{|\{k<n:B(k)=1\}|}=0.$$
So $H^f([(s,B)])=0$ by Corollary \ref{uniform_measure}. $H^f(\mathcal{M}_\Gamma)\leq H^f([ran(\Theta\upharpoonright\Gamma)])=0$ since $\Gamma$ is countable.
\\[2mm]
$\Leftarrow$: If for every $g\in\Gamma$, $f\geq^* g$, we construct a perfect tree $S$ such that $[S]\subseteq \mathcal{M}_\Gamma$ and $H^f([S])>0$. The idea is still to meet dense sets one by one without measure dropping, but this time we need to split within some Mathias condition, and ensure that the splitting for meeting the next dense set will not decrease the measure as well. We start from condition $(0,\omega)$, and we need the next claim to ensure the first step of the induction would work.

\begin{claim}
If for every $g\in\Gamma$, $f\geq^* g$, there is a gauge function $h$ such that
\begin{enumerate}
    \item $h(x)\geq x$ for every $x=2^{-n}$.
    \item For every $g\in\Gamma$, $h\geq^* g$.
    \item For every $A$, $H^h(A)>0$ if and only if $H^f(A)>0$.
\end{enumerate}
\end{claim}

\begin{proof}[Proof of the Claim]
Since $f(x)\geq^* x$, there is some $N\in\omega$ such that for all $n\geq N$,
$$f(2^{-n})\geq 2^{-n}.$$
Let $s=\inf \{f(2^{-n})/2^{-n}:n\in\omega\}$, then $s>0$, and for all $n\in\omega$
$$\frac{f(2^{-n})}{s}\geq 2^{-n}.$$
For every $A$, $H^{f/s}(A)>0$ if and only if $H^f(A)>0$. Let $h=f/s$.
\end{proof}

We can work with $h$ instead of $f$ now. We need some more definitions.

\begin{definition}
\begin{enumerate}
    \item For $a\in\omega$, a finite ($a$-)M-set $F$ is a finite subset of $P_M$ such that for every $(s,A),(t,B)\in F$, $|s|=|t|=a$ and $s,t$ incomparable.
    \item Given $a<b$, and finite $a$-M-set $F$, finite $b$-M-set $F^*$, $F^*$ extends $F$ if
    $$\{s:|s|=b\wedge \exists (t,B)\in F\ (s\supseteq t\wedge s-t\in B)\}=\{s:\exists A\ (s,A)\in F^*\}.$$
    \item For every finite M-set $F$, let tree $T_F$ be
    $$T_F=\{t\in 2^{<\omega}:\exists (s,A)\in F\ (t\subseteq s\vee (s\subseteq t\ \wedge\ t-s\in A))\}.$$
\end{enumerate}
See the picture below.

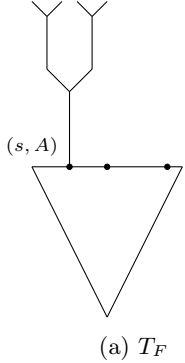
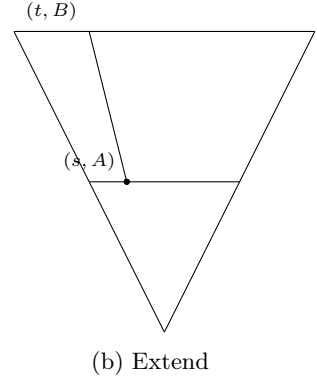
\begin{figure}[H]
     \centering
     \begin{subfigure}[b]{0.3\textwidth}

\begin{tikzpicture}
\draw (0,0) -- (-1,2);
\draw (0,0) -- (1,2);
\draw (-1,2) -- (1,2);
\filldraw (-0.5,2) circle (1pt);
\draw (-0.5,2) -- (-0.5,3);
\node[label={[xshift = -0.5cm, yshift= -0.1cm]\scriptsize{$(s,A)$}}] at (-0.5,2) {};
\filldraw (0,2) circle (1pt);
\filldraw (0.8,2) circle (1pt);
\foreach \x in {-0.3,0.3}
\draw (-0.5,3) -- (-0.5+\x, 3.3);
\foreach \x in {-0.3,0.3}
\draw (-0.5+\x,3.3) -- (-0.5+\x,4);
\foreach \x in {-0.3,0.3}
\foreach \y in {-0.2,0.2}
\draw (-0.5+\x,4) -- (-0.5+\x+\y, 4.2);
\end{tikzpicture}         
         \caption{$T_F$}
     \end{subfigure}
\hfill
     \begin{subfigure}[b]{0.3\textwidth}

\begin{tikzpicture}
\draw (0,0) -- (-1,2);
\draw (0,0) -- (1,2);
\draw (-1,2) -- (1,2);
\filldraw (-0.5,2) circle (1pt);
\node[label={[xshift = -0.5cm, yshift= -0.1cm]\scriptsize{$(s,A)$}}] at (-0.5,2) {};
\draw (-0.5,2) -- (-1,4);
\node[label={[xshift = -0.5cm, yshift= -0.1cm]\scriptsize{$(t,B)$}}] at (-1,4) {};
\draw (-1,2) -- (-2,4);
\draw (1,2) -- (2,4);
\draw (-2,4) -- (2,4);
\end{tikzpicture}
         \caption{Extend}
     \end{subfigure}
\end{figure}
\end{definition}

Recall from Proposition \ref{smalldim_lemma}, for every perfect tree $T$, there is a gauge function $f_T$ such that\\
for every open cover $\bigcup_i [\sigma_i]\supseteq [T]$,
$$\sum_{i\in\omega} f_T(2^{-|\sigma_i|})\geq 1.$$
We will prove the following one step lemma.

\begin{claim}\label{one-step-M}
Let $h$ satisfy our assumption. For every finite M-set $F\in\Gamma$ that $h\geq f_{T_F}$ and every $x\in\Gamma$ that codes a density operator $\Theta$, there is a finite M-set $F^*\in\Gamma$ extending $F$ such that
\begin{enumerate}
    \item $h\geq f_{T_{F^*}}$.
    \item for every $(s,A)\in F^*$, there exists $(t,B)\in ran(\Theta)$ such that $(s,A)\leq_M (t,B)$.
\end{enumerate}
\end{claim}

\begin{proof}[Proof of the claim]
We build a sequence of finite M-set $\{H_n:n\in\omega\}$. Let $F$ be a finite $a$-M-set.
$$H_n=\{\Theta(s,A):\exists (t,B)\in F (s=t\cup B\upharpoonright (0, |t|+n) \wedge\ A=B-s)\}.$$
So $H_0=F$. Let $f_n:=f_{T_{H_n}}$, $f(x):=\max f_n(x)$. There are two things we need to check.
\begin{enumerate}
    \item $\lim_{x\rightarrow 0+}f(x)=0$,
    \item $f\in\Gamma$.
\end{enumerate}
For every $n\in\omega$, $T_{H_n}$ and $T_F$ agree on the first $a+n$ levels, by 2. in Proposition \ref{smalldim_lemma}, $\forall k<a+n$, $f_n(2^{-k})=f_0(2^{-k})$. For every $\varepsilon>0$, there is some $K\in\omega$ such that $f_0(2^{-K})\leq\varepsilon$, so for every $n>K$, $f_n(2^{-K})\leq\varepsilon$, there exists some $L\geq K$ such that for every $n\leq k$, $f_n(2^{-L})\leq\varepsilon$, so for every $n$, $f_n(2^{-L})\leq\varepsilon$. We showed
$$\lim_{x\rightarrow 0+}f(x)=0.$$
The proof of $f\in\Gamma$ is by designing some algorithm, we omit it here.

So $f\in\Gamma$ is a gauge function, by assumption $h\geq^* f$. There is some $J>0$ such that for every $n\geq J$, $h(2^{-n})\geq f(2^{-n})$. Since for every $n<J$, $f_J(2^{-n})=f_0(2^{-n})$, $h$ dominates $f_J$. Let $F^*=H_J$, $F^*$ satisfies the above conditions.
\end{proof}

Fix an enumeration of elements of $\Gamma$ that code density operators as $\{x_n:n\in\omega\}$. Start with $F_0=\{(0,\omega)\}$, if $F_n$ is defined, let $x=x_n$, we get $F_{n+1}=F^*$ from applying the above claim. Let $T=\{s\in 2^{<\omega}: \exists n\exists (t,B)\in F_n(s\subseteq t)\}$.

\begin{claim}
$H^h([T])>0.$
\end{claim}

\begin{proof}[Proof of the claim]
Since $[T]$ is a compact set, for every $\{s_i:i\in\omega\}$ such that $[T]\subseteq \bigcup_i [s_i]$, there is a finite subset of it that covers $[T]$ as well. For such finite set $\{s_i:i<N\}$, let $m=\min \{k:\exists (s,A)\in F_k\forall i<N\ (|s_i|<|s|)\}$, then $\{s_i:i<N\}$ covers $T_{F_m}$ as well. Since $h$ dominates $f_{T_{F_m}}$, $\sum_{i<N} h(2^{-|s_i|})\geq 1$.
\end{proof}

\end{proof}

\subsection{Sacks Forcing}

Now we consider Sacks forcing, which is forcing with perfect trees. Sacks forcing is very similar to Mathias forcing if we view $(s,A)$ as a perfect tree with stem $s$ extended by possibly several $0$s and splits on the location marked by $A$. We will only give an outline of the proof here, pointing out what is different from Mathias forcing.

Let $\mathbb{P}_\mathcal{S}=(P_\mathcal{S},\leq_\mathcal{S})$ be the poset for Sacks forcing, where
    $$P_\mathcal{S}=\{T\subseteq 2^{<\omega}:T\text{ is a perfect tree}\},$$
    $$T\leq_\mathcal{S} R\longleftrightarrow T\subseteq R.$$
We say $x\in 2^\omega$ meets some $D\subseteq P_\mathcal{S}$ if there is some $T\in D$ such that $x\in [T]$. 

We identify a tree with a real in $2^\omega$ coding it, and say a density function $\Theta:P_\mathcal{S}\rightarrow P_\mathcal{S}$ is coded by a real $y$ if $\Theta$ can be extended to a continuous function from $2^\omega$ to $2^\omega$ and real $y$ codes this continuous function in some standard way.

\begin{definition}
Given a countable ideal $\Gamma$,
\begin{enumerate}
    \item $D\in\mathcal{D}_\Gamma$ if
    $$\exists x\in \Gamma (x\text{ codes }\Theta\ \wedge\ D=ran(\Theta\upharpoonright \Gamma)).$$
    \item $x\in 2^\omega$ is $\Gamma$-Sacks generic if $x$ meets every $D\in \mathcal{D}_\Gamma$. Let $S_\Gamma$ be the set of $\Gamma$-Sacks generics.
\end{enumerate}
\end{definition}

\begin{theorem}
\label{Sacks}
$H^f(S_\Gamma)>0$ if and only if for every gauge function $g\in\Gamma$, $f\geq^* g$.
\end{theorem}

\begin{proof}
$\Rightarrow:$ Suppose there is some gauge function $g\in\Gamma$ such that $f\not\geq^* g$. Given every perfect tree $T$, we can find a perfect tree $S\subseteq T$ computably in $g$ such that
$$H^g([S])=0.$$
The algorithm is basically the same as in Mathias forcing. In Mathias forcing, we work with condition $(s,A)$ and only need to specify the level of splits. Here, given a perfect tree, we always go to the leftmost path until we reach a level $n$ where the value $g(2^{-n})$ drops below $2^{-2k}$, then we allow the $k$th splitting above that level (i.e. each branch is allowed to split once), then follow the leftmost path on both of the splits until $g(2^{-n})$ drops below $2^{-2(k+1)}$, and allow splitting again.

The construction yields $\lim_{n\rightarrow\infty}g(2^{-n})\cdot |\{s\in S:|s|=n\}|=0$, and $f(2^{-n})<g(2^{-n})$ for infinitely many $n$, so
$$\liminf_{n\rightarrow\infty}f(2^{-n})\cdot |\{s\in S:|s|=n\}|=0.$$
So $H^f([S])=0$.

The algorithm gives us a continuous function $\Theta:2^\omega\rightarrow 2^\omega$ coded by a real in $\Gamma$ that serves as a density operator. $S_\Gamma\subseteq \bigcup_{S\in ran(\Theta\upharpoonright \Gamma)}[S]$, and $ran(\Theta\upharpoonright\Gamma)$ is countable, so 
$$H^f(S_\Gamma)\leq\sum_{S\in ran(\Theta\upharpoonright \Gamma)} H^f([S])=0.$$

$\Leftarrow:$ This direction is totally parallel to Mathias forcing. For every $a\in\omega$, we define a finite $a$-S-set $F$ to be a finite subset of $P_S$ such that for every $T_0,T_1\in P_S$, the stem of $T_0$ and the stem of $T_1$ are incomparable, and both have length $\geq a$. We have the one step lemma in Mathias forcing that given every finite S-set $F$, we can extend it to a finite S-set $F^*$ such that we meet the next dense set while $h$-measure of the tree is still large.
\end{proof}

\section{Discussions}

We are trying to find a correspondence between forcings and the characterizations of gauge functions for them. We have seen that the characterization is about dominance. There are also patterns of dominance for generic reals themselves.

The forcings we consider give generic reals in $2^\omega$. When we say dominance, we view every $x\in 2^\omega$ with $x^{-1}(\{1\})$ infinite as $\pi\in\omega^\omega$ defined as $\pi(0)=$ number of $0$'s before the first $1$, $\pi(n)=$ number of $0$'s between the $n$th $1$ and $n+1$th $1$.

\begin{definition}
For $x,y\in\omega^\omega$,
\begin{enumerate}
    \item $x$ dominates $y$ ($x\geq y$) if for every $n\in\omega$, $x(n)\geq y(n)$.
    \item $x$ eventually dominates $y$ ($x\geq^* y$) if there is some $N\in\omega$ such that for every $n\geq N$, $x(n)\geq y(n)$.
\end{enumerate}    
\end{definition}

All the propositions in this section are easy density arguments so we omit the proofs here.

\begin{proposition}
For every $x\in \mathcal{C}_\Gamma$ and every $y\in\Gamma$, $x\not\leq^* y$.
\end{proposition}

We can see from the above proposition and Theorem \ref{Cohen_dim} that the gauge functions making $\mathcal{C}_\Gamma$ measure positive are behaving similar to $\Gamma$-Cohen reals themselves. We look into Mathias forcing and Sacks forcing.

\begin{proposition}
For every $x\in \mathcal{M}_\Gamma$ and every $y\in\Gamma$, $x\geq^* y$.
\end{proposition}

\begin{proposition}
For every $x\in \mathcal{S}_\Gamma$, there is some $y\in\Gamma$ such that $x\leq y$.
\end{proposition}

From Theorem \ref{Mathias} and Theorem \ref{Sacks}, we can see that the behavior of the gauge functions making $\mathcal{M}_\Gamma$ and $\mathcal{S}_\Gamma$ measure positive are the same as $\Gamma$-Mathias reals, but $\Gamma$-Sacks reals have opposite behavior in terms of dominance. We do not know if it is still possible to find a general way to link the gauge profile of a set with the behavior of the reals in the set, if the set satisfies some closure properties and the reals in the set behave in similar ways.

\bibliographystyle{plain}
\bibliography{ref}

\end{document}